\documentclass{birkjour}
 \newtheorem{thm}{Theorem}[section]
 \newtheorem{cor}[thm]{Corollary}
 \newtheorem{lem}[thm]{Lemma}
 
 \theoremstyle{definition}
 \newtheorem{defn}[thm]{Definition}
 \theoremstyle{remark}
 \newtheorem{rem}[thm]{Remark}
 \newtheorem{ex}[thm]{Example}
 \numberwithin{equation}{section}
\usepackage{xcolor}
\usepackage[final]{changes}

\usepackage[hidelinks]{hyperref}
\usepackage{cleveref}

\allowdisplaybreaks[4]

\begin{document}

%
%
%
%
%

%
%
%
%

\title[$H$-tensional maps]{$H$-tensional maps}

\author{Bouazza Kacimi}
\address{%
University of Mascara\\
BP 305 Route de Mamounia\\
29000 Mascara\\
Algeria}
\email{bouazza.kacimi@univ-mascara.dz}

\author[Ahmed Mohammed Cherif]{Ahmed Mohammed Cherif}

\address{%
University of Mascara\\
BP 305 Route de Mamounia\\
29000 Mascara\\
Algeria}

\email{a.mohammedcherif@univ-mascara.dz}

\author{Mustafa \"{O}zkan}
\address{ Gazi University \\
Faculty of Sciences\\
Department of Mathematics \\
06500 Ankara\\
Turkey}
\email{ozkanm@gazi.edu.tr}



\date{January 1, 2004}

\begin{abstract}
	
\replaced{
In this paper, we study smooth maps, namely $H$-tensional,  between Riemannian manifolds whose tension fields are harmonic. Notice that, harmonic maps and minimal submanifolds form a special subclass of this notion. We obtain several nonexistence results for nonharmonic $H$-tensional maps and for nonminimal $H$-tensional submanifolds.}{	
In this paper, we study smooth maps between Riemannian manifolds whose tension fields are harmonic; we call such maps $H$-tensional. Harmonic maps and minimal submanifolds form special subclasses of this notion. We obtain several nonexistence results for nonharmonic $H$-tensional maps and for nonminimal $H$-tensional submanifolds.}
\end{abstract}
\maketitle
\section{Introduction}
\replaced{Let}{Given a smooth map} $\psi: (M,g)\rightarrow (N,h)$ \added{be a smooth map} between Riemannian manifolds of dimensions $m$ and $n$, respectively. The energy functional (also called the Dirichlet energy) of $\psi$, over a compact domain $\Omega$ of $M$, is defined by
\begin{equation}\label{eq1.1}
	E(\psi)=\frac{1}{2}\int_{\Omega}|d\psi|^{2}dv_g .
\end{equation}
The map $\varphi$ is called harmonic if it is a critical point of the energy functional \eqref{eq1.1}. The corresponding Euler-Lagrange equation is (see \cite{baird,eells2}):
\begin{equation*}
	\tau(\psi)=\operatorname{Tr}_{g}(\nabla d\psi)=\sum_{i=1}^{m}
	\{\nabla^{\psi}_{e_{i}}d\psi(e_{i})-d\psi(\nabla_{e_{i}}e_{i})\}=0,
\end{equation*}
where $\{e_{i}\}_{i=1}^{m}$ is a local orthonormal frame field of $(M,g)$, $\tau(\psi)$ is the tension field of $\psi$, $\nabla^{\psi}$ denotes the connection on the Riemannian vector bundle $\psi^{-1}TN\rightarrow M$ induced from the Levi-Civita connection $\nabla^{N}$ of $(N,h)$ and $\nabla$ the Levi-Civita connection of $(M,g)$. Thus the tension field $\tau(\psi)$ measures the failure of $\psi$ to be harmonic \replaced{and}{,} it vanishes precisely when $\psi$ is harmonic (see, for instance \cite{eells1}).

For an isometric immersion $\iota: M^{m}\rightarrow \mathbb{R}^{n}$ with \added{the} mean curvature vector $H$, \deleted{B.Y.} Chen (see \cite{chen2}) introduced the notion of  biharmonic submanifolds by requiring that $H$ is harmonic, i.e.\deleted{,} $\Delta H=\Delta (\frac{1}{m}\tau(\iota))=0$, where $\Delta$ is the Laplacian operator and proposed the following conjecture:

\noindent\textbf{Conjecture 1}. Any biharmonic submanifold of the Euclidean space is harmonic, i.e., minimal.

Note that \replaced{this}{the Chen} conjecture is not true for $\mathbb{R}^{3}$ with a non-flat Riemannian metric (see \cite{javaloyes}).

\replaced{Here, we remark  numerous partial results for this conjecture \cite{chen21,ou1}.}{Despite numerous partial results (see \cite{chen21}\cite{ou1}), this conjecture remains open.}

 Since \replaced{this}{Chen's} conjecture imposes no \added{an} assumption on the completeness of the submanifold, it is considered a problem in local differential geometry. Akutagawa and Maeta (see \cite{akutagawa}) reformulated the conjecture into a problem in global differential geometry:

\noindent\textbf{Conjecture 2}. Any complete biharmonic submanifold of the Euclidean space is minimal.

This conjecture has also attracted much attention, and many partial results exist \added{in the literature} (\replaced{for example,}{see} \cite{akutagawa}, \cite{maeta}), but the conjecture is still unresolved \added{as well as Conjecture~
	1}.

Motivated by Chen's notion of \added{a} biharmonic submanifold, in this note we introduce and study smooth maps between Riemannian manifolds whose tension fields are harmonic, that is, lie in the kernel of the rough Laplacian. We call such maps $H$-tensional. Harmonic maps are trivially $H$-tensional, but the converse need not hold\replaced{ which gives that}{; thus,} the class of $H$-tensional maps strictly extends the class of harmonic maps. This viewpoint also extends\added{,} naturally\added{,} to submanifolds: a submanifold $M$
is called $H$-tensional if the isometric immersion defining it is an $H$-tensional map. \replaced{Note that minimal}{Minimal} submanifolds are always $H$-tensional, and Chen's biharmonic submanifolds appear as a special case. \replaced{The aim of the paper is to develop some basic properties of $H$-tensional maps and $H$-tensional submanifolds. We also establish several nonexistence theorems for nonharmonic $H$-tensional maps and nonminimal $H$-tensional submanifolds. Finally, we }{
We develop basic properties of $H$-tensional maps and $H$-tensional submanifolds, establish several nonexistence theorems for nonharmonic $H$-tensional maps and nonminimal $H$-tensional submanifolds, and} propose two conjectures generalizing Conjectures~1 and 2.

Throughout this paper, we adopt the following conventions. For the curvature tensor, we use $R(X,Y)=[\nabla_{X},\nabla_{Y}]-\nabla_{[X,Y]}$, so that the Ricci operator is given by $\operatorname{Ric}(X)=\overset{m}{\underset{i=1}{\sum}}R(X,e_{i})e_{i}$. For functions $f\in C^{\infty}(M)$, the Laplacian is $\Delta f = \overset{m}{\underset{i=1}{\sum}}[(\nabla_{e_{i}}e_{i})f-e_{i}(e_{i}(f))]$. Also, we choose the convention $\Delta^{\psi}\xi=-\overset{m}{\underset{i=1}{\sum}}
\left(\nabla^{\psi}_{e_{i}}\nabla^{\psi}_{e_{i}}\xi-
\nabla^{\psi}_{\nabla_{e_{i}}e_{i}}\xi\right)$ for the rough Laplacian acting on sections $\xi\in\Gamma( \psi^{-1}TN)$.

\section{\added{The formal definition of the notion} ``$H$-tensional map'' \added{and some examples}}

\replaced{Although we have already said that smooth maps between Riemannian
	manifolds whose tension fields will be called $H$-tensional, the following is the formal definition.}{In this section, we introduce the notion of $H$-tensional maps
between Riemannian manifolds, and provide
some examples of such maps.}
\begin{defn}\label{def1.11}
Let $\psi: (M^{m},g)\longrightarrow (N^{n},h)$ be a smooth map between Riemannian manifolds. The map $\psi$ is called an $H$-tensional map if its tension field $\tau(\psi)$ satisfies
\begin{equation}\label{eq3.41}
	\Delta^{\psi} \tau(\psi)=0 .
\end{equation}
\end{defn}

\begin{rem}
\replaced{Every harmonic map is $H$-tensional by \eqref{eq3.41}. Hence,}{From \eqref{eq3.41}, it follows immediately that every harmonic map is} $H$-tensional map\replaced{s, naturally,}{. Hence, $H$-tensional maps naturally} generalize the notion of harmonic maps. 

Moreover, when the target manifold $N$ is the Euclidean space, \added{we obtain that} $H$-tensional maps coincide with biharmonic maps (see \cite{jiang}).
\end{rem}

\begin{ex}
Consider the Kelvin transformation given by
\begin{eqnarray*}
  \psi\;:\mathbb{R}^{m}\setminus\{0\} &\longrightarrow& \mathbb{R}^{m}\setminus\{0\} \\
  x &\longmapsto& \frac{x}{|x|^{l}} \, .
\end{eqnarray*}
Then, \replaced{we have that}{(see \cite{cherif1})}
\begin{equation*}
   \tau(\psi)= l(l-m)|x|^{-2-l}x,
\end{equation*}
and
\begin{equation*}
\Delta^{\psi}\tau(\psi)=-l(l-m)(-2-l)(-2+m-l)|x|^{-4-l}x 
\end{equation*}
\added{by \cite{cherif1}.} Hence, $\psi$ is nonharmonic $H$-tensional map if and only if $m=l+2$ or $l=-2$.
\end{ex}
\begin{ex}
Consider the hyperbolic $3$-space $$\mathbb{H}^3=\{(x,y,z)\in\mathbb{R}^3:\;z>0\}$$
\replaced{and let}{Let the identity map}
\begin{eqnarray*}
  I\;:\left(\mathbb{H}^3,z^{-2p}(dx^2+dy^2+dz^2)\right) &\longrightarrow& \left(\mathbb{H}^3,w^{-2}(du^2+dv^2+dw^2)\right) \\
  (x,y,z) &\longmapsto& (x,y,z),
\end{eqnarray*}
\added{be the identity map,} where $p\in\mathbb{R}$.

 The vector fields
\begin{equation*}
	e_{1}=z^{p}\frac{\partial}{\partial x},\;\;e_{2}=z^{p}\frac{\partial}{\partial y}
	,\;\;e_{3}=z^{p}\frac{\partial}{\partial z},
\end{equation*}
constitute an orthonormal basis of $\left(\mathbb{H}^3,z^{-2p}(dx^2+dy^2+dz^2)\right)$\replaced{. We have}{,
for which, we have the Lie brackets:}

\added{$\bullet$ the components of the Levi-Civita connection:}
\begin{equation*}
	[e_{1},e_{2}]=0,\;\;[e_{1},e_{3}]=-pz^{p-1}e_{1},\;\;[e_{2},e_{3}]=-pz^{p-1}e_{2};
\end{equation*}

\replaced{$\bullet$ the components of the Levi-Civita connection:}{The components of the Levi-Civita connection are determined by}
\begin{align}\label{eq4.1}
	\nabla_{e_{1}}e_{1}&=pz^{p-1}e_{3},           &  \nabla_{e_{1}}e_{2} &=0,              &  \nabla_{e_{1}}e_{3}&=-pz^{p-1}e_{1},\nonumber\\
	\nabla_{e_{2}}e_{1}&=0,         &  \nabla_{e_{2}}e_{2}&=pz^{p-1}e_{3},   &  \nabla_{e_{2}}e_{3}&=-pz^{p-1}e_{2},\\
	\nabla_{e_{3}}e_{1}&=0,   &  \nabla_{e_{3}}e_{2}&=0,          &  \nabla_{e_{3}}e_{3}&=0.\nonumber
\end{align}
The vector fields
\begin{equation*}
	\bar{e}_{1}=w\frac{\partial}{\partial u},\;\;\bar{e}_{2}=w\frac{\partial}{\partial v}
	,\;\;\bar{e}_{3}=w\frac{\partial}{\partial w},
\end{equation*}
constitute an orthonormal basis of $\left(\mathbb{H}^3,w^{-2}(du^2+dv^2+dw^2)\right)$. A straightforward computation yields \added{that}
\begin{align*}
	\tau(I)&=(1-p)z^{2p-2}\bar{e}_{3},\\
	\Delta^I(\tau(I))&=2(p-1)(p^{2}-4p+2)z^{4p-4}\bar{e}_{3}.
\end{align*} 
\replaced{Hence,}{so} for $p=2\pm\sqrt{2}$, \replaced{we obtain that $I$ is a}{$I$ is} nonharmonic $H$-tensional map.
\end{ex}

\section{Nonexistence results of nonharmonic $H$-tensional maps}

In this section, we establish several nonexistence results of nonharmonic $H$-tensional maps. 

We begin by recalling that a smooth function $f$ on a Riemannian
manifold $(M,g)$ is said to be strongly convex
if its Hessian satisfies
\begin{equation*}
    (\operatorname{Hess} f)(X,X)> 0\;\;\;\;\text{ for all}\; X\in\Gamma(TM),
\end{equation*}
(see \cite{xin}).

Before proceeding, we introduce a class of vector fields that will play a key role in our subsequent results.
\begin{defn}
A vector field $\xi$ on a Riemannian manifold $(M,g)$ is said to be of rough-type
if it satisfies
\begin{equation*}
  \nabla_{X}\nabla_{X}\xi-\nabla_{\nabla_{X}X}\xi=0,\;\;\;\forall X\in\Gamma(TM).
\end{equation*}
\end{defn}
To illustrate this concept, we describe all rough-type vector fields on the Euclidean
space.
\begin{thm}\label{Rough}
Let $\xi=\sum_{j=1}^m f_{j}\frac{\partial}{\partial x_{j}}$ be a vector field on the Euclidean
space $(\mathbb{R}^{m}, <,>)$ with the canonical Euclidean metric $<,>$ and the usual global coordinates $(x_{1},\cdots,x_{m})$, such that $f_{j}=f_j(x)=f_{j}(x_{1},\cdots,x_{m})$. Then, \added{ the following statements are equivalent:}
\begin{enumerate}
	\item $\xi$ is of rough-type vector field,
	\item
\begin{equation}\label{eqr}
    f_j(x)=\sum_{P\subseteq\{1,\dots,m\}} c_{j,P}\,\Big(\prod_{i\in P} x_i\Big),
\end{equation}
where the sum is taken on all subset $P\subseteq\{1,\dots,m\}$, if $P=\emptyset$ then the product is by convention $1$ and the $c_{j,P}$ are constants.
\end{enumerate}
\end{thm}
\begin{proof}
\added{$(1) \Rightarrow (2).$} By a direct calculation, we
get that $\xi$ is of rough-type vector field if
\begin{equation}\label{eqrr}
   \frac{\partial^{2}f_j(x)}{\partial x_k^2}=0,\qquad\text{for all }j\in\{1,\dots,m\},\; k\in\{1,\dots,m\}.
\end{equation}
We prove \eqref{eqr} by induction\replaced{. W}{, w}e fix $f=f_{j}$.

If $m=1$. Then $\frac{d^{2}f}{dx_{1}^{2}}=0$ so $f(x_{1})=a x_{1}+ b$, where $a$\replaced{ and}{,} $b$  are constants\replaced{. T}{, t}hus \eqref{eqr} holds because the only possible subsets $P$ of $\{1\}$ are $\emptyset$ and $\{1\}$.

Assume \added{that} the claim  holds for $m-1$. Let $ \bar{x}=(x_{1},\cdots,x_{m-1})$\replaced{. By}{, from} \eqref{eqrr}, we have \added{obtain that} (for each fixed $\bar{x}$))
\begin{equation*}
\frac{ \partial^2f(\bar{x},x_m)}{\partial x_m} = 0.
\end{equation*}
Then
\begin{equation}\label{eqrrr}
f(\bar{x},x_m) = J(\bar{x})\,x_m + K(\bar{x}),
\end{equation}
where both $J$ and $K$ are smooth functions depending on $\bar{x}$.
Differentiate \eqref{eqrrr} twice with respect to $x_{i}$ (treating  $x_{m}$ as a parameter) and apply the remaining PDEs
\begin{equation}
   \frac{\partial^{2}f(x)}{\partial x_k^2}=0,\qquad\text{for all }k\in\{1,\dots,m-1\},
\end{equation}
we obtain \added{that}
\begin{equation*}
    0 = \frac{ \partial^2f(\bar{x},x_m)}{\partial x_i^2} = \frac{\partial^2J(\bar{x})}{\partial x_i^2}\,x_m + \frac{\partial^2K(\bar{x})}{\partial x_i^2}.
\end{equation*}
Thus
\begin{equation*}
    \frac{\partial^2J(\bar{x})}{\partial x_i^2}=0\qquad\text{and}\qquad \frac{\partial^2K(\bar{x})}{\partial x_i^2}=0.
\end{equation*}
By the induction hypothesis, we get \added{that}
\begin{equation*}
J(\bar{x})=\sum_{P\subseteq\{1,\dots,m-1\}}  \alpha_P \Big(\prod_{i\in P} x_i\Big),\;\text{and}\;
K(\bar{x})=\sum_{P\subseteq\{1,\dots,m-1\}} \beta_P \Big(\prod_{i\in P} x_i\Big)
\end{equation*}
where both $\alpha_P$ and $\beta_P$ are constants. Therefore by \eqref{eqrrr} we find
\begin{equation*}
\begin{aligned}
f_{j}(x) &= \sum_{P\subseteq\{1,\dots,m-1\}} \alpha_P \Big(\prod_{i\in P} x_i\Big)\,x_m
+ \sum_{P\subseteq\{1,\dots,m-1\}} \beta_P \Big(\prod_{i\in P} x_i\Big)\\
&=\sum_{P\subseteq\{1,\dots,m\}} c_{j,P}\,\Big(\prod_{i\in P} x_i\Big).
  \end{aligned}
\end{equation*}
\replaced{$(2) \Rightarrow (1).$ It is easy to see that }{Conversely,} if \eqref{eqr} holds\added{,} then $\xi=\sum_{j=1}^m f_{j}\frac{\partial}{\partial x_{j}}$ is of rough-type vector field on $(\mathbb{R}^{m}, <,>)$ \added{as desired}.
\end{proof}

\begin{rem}
Since the rough-type vector fields coincide with the Jacobi-type vector fields on $(\mathbb{R}^{m}, <,>)$ (cf. \cite{deshmukh}), \replaced{we can easily say that}{then} Theorem~\ref{Rough} remains true for such classes of vector fields.
\end{rem}
\begin{ex}
Let $\xi=\sum_{j=1}^3 f_j(x_1,x_2,x_3)\frac{\partial}{\partial x_{j}}$ be a vector field on $(\mathbb{R}^{3}, <,>)$. Since $m=3$,  the set of subsets of $\{1,2,3\}$ has $2^{3}=8$ elements:
\begin{equation*}
    \emptyset,\;\{1\},\;\{2\},\;\{3\},\;\{1,2\},\;\{1,3\},\;\{2,3\},\;\{1,2,3\}.
\end{equation*}
So by Theorem~\ref{Rough}, $\xi$ is of rough-type if and only if
\begin{equation*}
\begin{aligned}
f_j(x_1,x_2,x_3) &= c_{j,\emptyset}
+ c_{j,\{1\}}\,x_1
+ c_{j,\{2\}}\,x_2
+ c_{j,\{3\}}\,x_3
+ c_{j,\{1,2\}}\,x_1x_2\\
&\quad+ c_{j,\{1,3\}}\,x_1x_3+ c_{j,\{2,3\}}\,x_2x_3
+ c_{j,\{1,2,3\}}\,x_1x_2x_3.
  \end{aligned}
\end{equation*}
\end{ex}
\begin{ex}
The
position vector field $\xi=\sum_{j=1}^m x_{j}\frac{\partial}{\partial x_{j}}$ of $\mathbb{R}^{m}$
is of rough-type.
\end{ex}

\begin{thm}\label{Th4.999}
Let $(M,g$) be a compact\deleted{,} orientable Riemannian manifold and \deleted{let} $(N, h)$ be a Riemannian manifold that possesses a strongly convex function
$f$. If the vector field $\operatorname{grad}^{N}f$ is of rough-type on $(N, h)$, then every $H$-tensional map $\psi: (M,g)\rightarrow (N,h)$ is constant.
\end{thm}
\begin{proof}
Let $\{e_{i}\}_{i=1}^{m}$ be a normal orthonormal frame on $M$ at $x\in M$. We set
\begin{equation*}
    \rho=h((\operatorname{grad}^{N}f)\circ\psi,\tau(\psi)).
\end{equation*}
\replaced{Then}{then}
\begin{equation}\label{eqn1}
    e_{i}(\rho)=h(\nabla_{e_{i}}^{\psi}(\operatorname{grad}^{N}f)\circ\psi,\tau(\psi))+
    h((\operatorname{grad}^{N}f)\circ\psi,\nabla_{e_{i}}^{\psi}\tau(\psi)),
\end{equation}
from \eqref{eqn1}, \replaced{which implies that}{we get}
\begin{align}
  \Delta(\rho) &=-h(\nabla_{e_{i}}^{\psi}\nabla_{e_{i}}^{\psi}(\operatorname{grad}^{N}f)\circ\psi,\tau(\psi))-2
    h(\nabla_{e_{i}}^{\psi}(\operatorname{grad}^{N}f)\circ\psi,\nabla_{e_{i}}^{\psi}\tau(\psi))\nonumber\\
   &\quad -h((\operatorname{grad}^{N}f)\circ\psi,\nabla_{e_{i}}^{\psi}\nabla_{e_{i}}^{\psi}\tau(\psi)) \label{eqn2}
\end{align}
\replaced{ because of}{by \eqref{eqn2}, and}  the fact that $\psi$ is $H$-tensional map from a compact orientable Riemannian manifold $(M,g)$ \added{and \eqref{eqn2}.} \replaced{Hence}
{, we find:}
\begin{equation}\label{eqn3}
    \Delta(\rho) =-h(\nabla_{e_{i}}^{\psi}\nabla_{e_{i}}^{\psi}(\operatorname{grad}^{N}f)\circ\psi,\tau(\psi)).
\end{equation}
As $\operatorname{grad}^{N}f$ is a rough-type vector field on $(N,h)$, by \eqref{eqn3}, we yield \added{that}
\begin{equation}\label{eqn6}
    \Delta(\rho) =-h(\nabla_{\tau(\psi)}^{N}(\operatorname{grad}^{N}f),\tau(\psi)).
\end{equation}
\replaced{Hence, by}{From} \eqref{eqn6}, $\operatorname{Hess}f>0$, and \deleted{the divergence theorem, we get} $ \tau(\psi)= 0$, i.e., $\psi$ is
harmonic map \added{by the divergence theorem.} \replaced{Thus,}{thus} by Corollary~1.4.4 in \cite{xin} we deduce that $\psi$ is constant.
\end{proof}
Now, consider the function $f$ on $(\mathbb{R}^{n}, <,>)$ defined by
\begin{equation*}
    f(x)=\frac{1}{2}|x|^{2},\;x\in\mathbb{R}^{n}.
\end{equation*}
Since $\operatorname{grad}f$ is the position vector field in $\mathbb{R}^{n}$\replaced{, we obtain that}{ then} it is a rough-type vector field\replaced{. M}{, m}oreover it satisfies $\operatorname{Hess}f=<,>$ on $\mathbb{R}^{n}$\replaced{. Therefore,}{, thus} by Theorem~\ref{Th4.999} \replaced{yields}{we find the following result}.
\begin{cor}\label{Co1000}
Let $(M,g$) be a compact, orientable Riemannian manifold. Then, every $H$-tensional map $\psi: (M,g)\rightarrow (\mathbb{R}^{n}, <,>)$ is constant.
\end{cor}
\begin{rem}
Corollary~\ref{Co1000} has been proved by the second author in \cite{cherif2} in the setting of the biharmonic maps.
\end{rem}

Now, we need the following Lemma for \replaced{the second main theorem of this section.}{later use}
\begin{lem}\label{lem100}
Let $\psi: (M^{m},g)\longrightarrow (N^{n},h)$ be an $H$-tensional map from a non-compact complete Riemannian manifold $(M, g)$ into a Riemannian manifold $(N, h)$ and let $q$ be a real constant satisfying $2\leq q<\infty$. If\added{,} for such a $q$,
\[
\int_M |\tau(\psi)|^q \, dv_g < \infty,
\]
then $\nabla^\psi _X \tau(\psi) = 0$ for any vector field $X$ on $M$. In particular, $|\tau(\psi)|$ is constant.
\end{lem}

\begin{proof}
For a fixed point $x_0 \in M$, and for every $0 < r < \infty$, we take a cut-off function $\beta$ on $M$ satisfying
\[
\begin{cases}
0 \leq \beta(x) \leq 1, & x \in M, \\
\beta(x) = 1, & x \in B_r(x_0), \\
\beta(x) = 0, & x \not\in B_{2r}(x_0), \\
|\operatorname{grad}^{M}\beta| \leq \frac{C}{r}, & x \in M,
\end{cases}
\]
where $C$ is a constant independent of $r$\replaced{ and}{, where} $B_{r}(x_0)$\replaced{,}{and} $B_{2r}(x_0)$ are the balls centered at a fixed point $x_0 \in M$ with
radius $r$ and $2r$ respectively. \replaced{Then, we obtain that}{From \eqref{eq3.41}, we have}
\begin{equation}\label{7}
\int_M h(\Delta^\psi \tau(\psi), \beta^2 |\tau(\psi)|^{q - 2} \tau(\psi) ) dv_g
= 0
\end{equation}
\added{by \eqref{eq3.41}, and hence}
\deleted{By \eqref{7}, we have}
\begin{align}
0 &= \int_M h(\Delta^\psi \tau(\psi), \beta^2 |\tau(\psi)|^{q - 2} \tau(\psi) ) dv_g \nonumber\\
&= \int_M h( \nabla^\psi \tau(\psi), \nabla^\psi(\beta^2 |\tau(\psi)|^{q - 2} \tau(\psi)) ) dv_g\nonumber\\
  &= \int_M \sum_{i=1}^m h( \nabla_{e_{i}}^\psi \tau(\psi), (e_{i}\beta^{2})|\tau(\psi)|^{q - 2} \tau(\psi)+\beta^{2}
  e_{i}\{(|\tau(\psi)|^{2})^{q - 2})\}\tau(\psi)\nonumber\\
  &\quad+\beta^{2}|\tau(\psi)|^{q - 2}\nabla_{e_{i}}^\psi \tau(\psi)) dv_g\nonumber\\
  &= \int_M \sum_{i=1}^m h( \nabla_{e_{i}}^\psi \tau(\psi), 2\beta(e_{i}\beta)|\tau(\psi)|^{q - 2} \tau(\psi)dv_g\nonumber\\
  &\quad+\int_M \sum_{i=1}^m \beta^{2}(q - 2)|\tau(\psi)|^{q - 4}(h( \nabla_{e_{i}}^\psi \tau(\psi),\tau(\psi))^{2} dv_g\nonumber\\
 &\quad+\int_M \sum_{i=1}^m h( \nabla_{e_{i}}^\psi \tau(\psi),\beta^{2}|\tau(\psi)|^{q - 2}\nabla_{e_{i}}^\psi \tau(\psi)) dv_g \label{8}
  \end{align}
\replaced{by \eqref{7}. Since}{Since}
$$\beta^{2}(q - 2)|\tau(\psi)|^{q - 4}(h( \nabla_{e_{i}}^\psi \tau(\psi),\tau(\psi))^{2}\geq 0,$$
the equation \eqref{8} becomes
\begin{align}
0&\geq \int_M \sum_{i=1}^m h( \nabla_{e_{i}}^\psi \tau(\psi), 2\beta(e_{i}\beta)|\tau(\psi)|^{q - 2} \tau(\psi)dv_g\nonumber\\
 &\quad+\int_M \sum_{i=1}^m h( \nabla_{e_{i}}^\psi \tau(\psi),\beta^{2}|\tau(\psi)|^{q - 2}\nabla_{e_{i}}^\psi \tau(\psi)) dv_g \label{88}
  \end{align}
By using Young's inequality, that is,
\[
\pm 2 \langle V, W \rangle \leq \varepsilon |V|^2 + \frac{1}{\varepsilon} |W|^2,
\]
for all positive $\varepsilon$, we \replaced{have that}{find}
\begin{align}
&-2 \int_M \sum_{i=1}^m h( \nabla_{e_{i}}^\psi \tau(\psi), \beta(e_{i}\beta)|\tau(\psi)|^{q - 2} \tau(\psi)) dv_g\nonumber\\
&= -2 \int_M \sum_{i=1}^m h(  (e_{i}\beta)|\tau(\psi)|^{\frac{q}{2} - 1} \tau(\psi),\beta|\tau(\psi)|^{\frac{q}{2} - 1}\nabla_{e_{i}}^\psi \tau(\psi)) dv_g \nonumber\\
&\leq
2\int_M |\operatorname{grad}^{M}\beta|^2 |\tau(\psi)|^q dv_g + \frac{1}{2} \int_M \sum_{i=1}^m \beta^{2}|\tau(\psi)|^{q - 1}|\nabla_{e_{i}}^\psi \tau(\psi)|^{2}dv_g.\label{9}
\end{align}
Substituting \eqref{9} \added{into} \eqref{88}, \added{we get that} \deleted{this into \eqref{88}, we get}
\begin{align}
\int_M \sum_{i=1}^m \beta^{2}|\tau(\psi)|^{q - 2}|\nabla_{e_{i}}^\psi \tau(\psi)|^{2}dv_g&\leq 4\int_M |\operatorname{grad}^{M}\beta|^2 |\tau(\psi)|^q dv_g\nonumber\\
&\leq \int_M\frac{4C^2}{r^2}  |\tau(\varphi)|^q dv_g.\label{10}
\end{align}
The right hand side of \eqref{10} goes to zero as $r \to \infty$ because the integral on the right-hand side is finite by \added{the} assumption and the left
hand side of \eqref{10} goes to
\begin{equation*}
    \int_M \sum_{i=1}^m |\tau(\psi)|^{q - 2}|\nabla_{e_{i}}^\psi \tau(\psi)|^{2}dv_g
\end{equation*}
if $r \to \infty$, since $\beta=1$ on $B_r(x_0)$. Thus, we obtain \added{that}
\begin{equation*}
    \int_M \sum_{i=1}^m |\tau(\psi)|^{q - 2}|\nabla_{e_{i}}^\psi \tau(\psi)|^{2}dv_g=0,
\end{equation*}
\replaced{which implies that}{Therefore, we obtain}
$|\tau(\psi)|$ is constant and $\nabla^\psi _X \tau(\varphi) = 0$ for any vector field $X$ on $M$.
\end{proof}

\begin{thm}
Let $\psi: (M^{m},g)\longrightarrow (N^{n},h)$ be an $H$-tensional map from a non-compact complete Riemannian manifold $(M, g)$ into a Riemannian manifold $(N, h)$ and let $q$ be a real constant satisfying $2\leq q<\infty.$
\begin{itemize}
\item[(i)] If
\[
\int_M |\tau(\psi)|^q \ dv_g < \infty, \quad \text{and} \quad \int_M |d\psi|^2 \, dv_g < \infty,
\]
then $\psi$ is harmonic.
\item[(ii)] If $\operatorname{Vol}(M) = \infty$ and
\[
\int_M |\tau(\psi)|^q \, dv_g < \infty,
\]
then $\psi$ is harmonic.
\end{itemize}
\end{thm}
\begin{proof}
\replaced{By}{From} Lemma~\ref{lem100}, we have \added{that} $\nabla^\psi _X \tau(\varphi) = 0$ for any $X\in\mathfrak{X}(M)$ and $|\tau(\psi)|$ is constant.

\noindent $(i)$. Define a $1$-form $\theta$ on $M$ by
\[
\theta(X) := |\tau(\psi)|^{\frac{q}{2} - 1} h(d\psi(X), \tau(\psi)),\;\;\;X\in\mathfrak{X}(M).
\]
Since $\int_M |d\psi|^2 \, dv_g < \infty$ and $\int_M |\tau(\psi)|^q \, dv_g < \infty,$ we yield \added{that}
\begin{align}
\int_M |\theta|dv_g&=\int_M \Big (\sum_{i=1}^m|\theta(e_{i})|\Big )^2 dv_g\nonumber\\
&\leq \int_M|\tau(\psi)|^\frac{q}{2}|d\psi| dv_g\nonumber\\
&\leq \Big (\int_M |d\psi|^{2}dv_g\Big )^2 \Big (\int_M |\tau(\psi)|^{q}dv_g\Big )^2<\infty.\label{eq100}
\end{align}
Now, we have \added{that} $- \delta \theta= \sum_{i=1}^m \left( \nabla_{e_i} \theta\right)(e_i) =|\tau(\varphi)|^{\frac{q}{2} + 1}$ \added{by} \cite{maeta}. Since $|\tau(\psi)|$ is constant and $\int_M |\tau(\psi)|^q \, dv_g < \infty,$ the function $\delta \theta$ is integrable over $M$\replaced{, which implies, applying  Gaffney's theorem (for example, see  \cite{gaffney} and \cite{nakauchi}) for the $1$-form $\theta$, that}{. From this and \eqref{eq100}, we can apply Gaffney's theorem see \cite{gaffney} and \cite{nakauchi} for the $1$-form $\theta$. Thus we obtain}
\[
0 = \int_M (-\delta \theta) \, dv_g = \int_M |\tau(\psi)|^{\frac{q}{2} - 1} \, dv_g,
\]
\replaced{by \eqref{eq100}. Hence}{which implies} $|\tau(\psi)| = 0$, \replaced{and hence}{so} $\psi$ is harmonic.

\noindent $(ii)$.  If $\operatorname{Vol}(M) = \infty$ and $|\tau(\psi)| \ne 0$, then
\[
\int_M |\tau(\psi)|^q \, dv_g = |\tau(\psi)|^q \cdot \operatorname{Vol}(M) = \infty,
\]
contradicting the assumption. So $|\tau(\psi)| = 0$ and hence $\tau(\psi) = 0$, i.e., $\psi$ is harmonic.
\end{proof}
\section{\added{A new member of Riemannian's manifold folks:} $H$-tensional submanifolds}
\deleted{This section is devoted to study $H$-tensional maps in the context of Riemannian submanifolds.}
\begin{defn}
A (connected) Riemannian submanifold $\iota: (M^{m},g)\longrightarrow (N^{n},h)$ of a Riemannian manifold is \replaced{said to be}{called} $H$-tensional
if the isometric immersion $\iota$ is an $H$-tensional map.
\end{defn}
Clearly, minimal submanifolds are $H$-tensional. Let $\iota: (M^{m},g)\longrightarrow (N^{n},h)$ be a Riemannian submanifold of a Riemannian manifold. We shall denote by $B$, $A$, and $\Delta^{\bot}$ the second fundamental form, the shape operator and the Laplacian on the normal bundle of $M^{m}$, respectively.
\begin{thm}\label{Th1000}
A Riemannian submanifold $\iota: (M^{m},g)\longrightarrow (N^{n},h)$ of a Riemannian manifold is $H$-tensional if
and only if
\begin{equation}\label{eq12000}
\begin{cases}
\strut \Delta^{\bot}H+ \operatorname{Tr}(B(\cdot,A_{H}(\cdot)))
    =0,\\
\frac{m}{4}\operatorname{grad}(|H|^{2})+\operatorname{Tr}(A_{\nabla_{\cdot}^{\bot}H}\cdot)
+\frac{1}{2}\left[\operatorname{Tr} R^N(d\iota(\cdot), H)d\iota(\cdot) \right]^\top=0,
\end{cases}
\end{equation}
where $R^{N}$ is the curvature
operator of $(N, h)$.
\end{thm}
\begin{proof}
Let $\{e_i\}_{1\leq i\leq m}$ be a local geodesic orthonormal frame  at point $x$ in $M$. Then, calculating at $x$, and using the Gauss and Weingarten formulas, we have \added{that}
\begin{gather}\label{eq1000}
\begin{aligned}
  \Delta^{\iota} \tau(\iota)&= m\Delta^{\iota}H=-m\underset{i=1}{\overset{m}{\sum}}(\nabla^{\varphi}_{e_i}\nabla^{\varphi}_{e_i}H)=
- m\underset{i=1}{\overset{m}{\sum}}(\nabla^{\varphi}_{e_i}(-A_{H}e_i+\nabla_{e_i}^{\bot}H ))\\
 &=- m\underset{i=1}{\overset{m}{\sum}}(-\nabla_{e_i}A_{H}e_i- B(e_i,A_{H}e_i)-A_{\nabla_{e_i}^{\bot}H}e_i+
\nabla^{\bot}_{e_i}\nabla^{\bot}_{e_i}H)\\
 &=m[\Delta^{\bot}H+ \operatorname{Tr}(B(\cdot,A_{H}(\cdot)))]+m\sum_{i=1}^m [A_{\nabla_{e_i}^{\bot}H}e_i))+\nabla_{e_i} A_H(e_i)].
\end{aligned}
\end{gather}
Moreover, at $x$ \replaced{by}{one has} \cite{dong}\added{,  one has that }
\begin{equation}\label{eq1001}
    \sum_{i=1}^m\nabla_{e_i} A_H(e_i)=\frac{m}{2}\operatorname{grad}(|H|^{2})+\sum_{i=1}^m[A_{\nabla_{e_i}^{\bot}H}e_i+\left[R^N(d\iota(e_i), H)d\iota(e_i) \right]^\top]
\end{equation}
Substituting \eqref{eq1001} in \eqref{eq1000} and comparing the normal and the tangential
components \replaced{finishes the proof}{we obtain the theorem}.
\end{proof}
\begin{rem}\label{Rq4.101}
The first condition of \eqref{eq12000} implies immediately that every $H$-tensional submanifold is minimal if it has \added{a}
parallel mean curvature vector, i.e., $\nabla^{\bot}H = 0$.
\end{rem}
\begin{thm}\label{Th4.100}
A Riemannian submanifold $\iota: (M^{m},g)\longrightarrow (N^{n}(c),h)$ in a space of constant sectional curvature $c$ is
$H$-tensional if and only if
\begin{equation}\label{eq12 11}
\begin{cases}
\strut \Delta^{\bot}H+ \operatorname{Tr}(B(\cdot,A_{H}(\cdot)))
    =0,\\
\frac{m}{4}\operatorname{grad}(|H|^{2})+\operatorname{Tr}(A_{\nabla_{\cdot}^{\bot}H}\cdot)=0.
\end{cases}
\end{equation}
\end{thm}
\begin{proof}
We have \added{that}
\begin{align*}
\operatorname{Tr}_{g} R^N(d\iota(\cdot), H)d\iota(\cdot)&=\overset{m}{\underset{i=1}{\sum}}R^{N}(d\iota(e_{i}),H)d\iota(e_{i})\\
 &= -c\;\overset{m}{\underset{i=1}{\sum}}h(d\iota(e_{i}), d\iota(e_{i}))H\\
 &= -m\; c H
\end{align*}
Therefore
\begin{equation*}
    \left[\operatorname{Tr} R^N(d\iota(\cdot), H)d\iota(\cdot) \right]^\top=0.
\end{equation*}
Now, by Theorem~\ref{Th1000}, we conclude \added{the claim}.
\end{proof}
\begin{rem}\label{Rq4.55}
If we consider the $H$-tensional equations \eqref{eq12 11} in the Euclidean space, we recover Chen's notion of biharmonic submanifolds, so the
two definitions coincide.
\end{rem}
\begin{thm}\label{Th4.1}
A hypersurface $\iota: M^{m}\longrightarrow N^{m+1}$ with \added{the} mean curvature vector $H=\alpha e_{m+1}$\deleted{,} is $H$-tensional if and only if
\begin{equation}\label{eq12 2}
\begin{cases}
\strut \Delta \alpha+\alpha|A|^{2}
    =0,\\
\frac{m}{4}\operatorname{grad}\alpha^{2}+A(\operatorname{grad}\alpha)-\frac{\alpha}{2}(\operatorname{Ric}^{N}(e_{m+1}))^{\top}=0,
\end{cases}
\end{equation}
where $\operatorname{Ric}^{N}$ denotes the Ricci operator.
\end{thm}
\begin{proof}
Sine $H=\alpha e_{m+1}$, we have \added{that}
\begin{align*}
\Delta^\perp H &= (\Delta \alpha)\, e_{m+1}, \\
\mathrm{Tr}\, B(A_H(\cdot), \cdot)& = \sum_{i=1}^m  B(A_H(e_i), e_i) = \alpha |A|^2 e_{m+1}, \\
\left( \mathrm{Trace}\, R^N(d\iota(\cdot), H) d\iota(\cdot) \right)^\top &= - \alpha \left( \mathrm{Ric}^N(e_{m+1}) \right)^\top, \\
\mathrm{Tr}\, A_{\nabla^\perp_{(\cdot)} H} (\cdot) &= \sum_{i=1}^m A_{\nabla^\perp_{e_i} H}(e_i) = A(\mathrm{grad}\, \alpha).
\end{align*}
Substituting these into the $H$-tensional equations \eqref{eq12000} yields the desired system.
\end{proof}
\begin{cor}\label{Th4.1000}
A hypersurface in an Einstein space $N^{m+1}$ with \added{the} mean curvature vector $H=\alpha e_{m+1}$ is $H$-tensional if
and only if
\begin{equation}\label{eq12 222}
\begin{cases}
\strut \Delta \alpha+\alpha|A|^{2}
    =0,\\
\frac{m}{4}\operatorname{grad}\alpha^{2}+A(\operatorname{grad}\alpha)=0.
\end{cases}
\end{equation}
\end{cor}
\begin{cor}\label{Th4.100}
A hypersurface $\iota: M^{m}\longrightarrow N^{m+1}(c)$ in a space of constant sectional curvature $c$ is
$H$-tensional if and only if
\begin{equation}\label{eq12 200}
\begin{cases}
\strut \Delta \alpha+\alpha|A|^{2}
    =0,\\
\frac{m}{4}\operatorname{grad}\alpha^{2}+A(\operatorname{grad}\alpha)=0.
\end{cases}
\end{equation}
\end{cor}
\section{Nonexistence results of nonminimal $H$-tensional submanifolds}
The goal of this section is to prove several nonexistence results of nonminimal $H$-tensional submanifolds.
\begin{thm}\label{Th45}
Let $\iota: (M^{m},g)\longrightarrow (N^{n},h)$ be a Riemannian submanifold \replaced{with}{such that} $|H|=$ constant. Then $M$ is $H$-tensional if and only if it is minimal.
\end{thm}
\begin{proof}
Assume that $M$ is $H$-tensional, by using the Weitzenb\"{o}ck formula, we obtain \added{that}
\begin{align}
  \frac{1}{2}\Delta|H|^{2}&= h(\Delta^{\iota}H,H)-|\nabla^{\iota}H|^{2},\nonumber\\
 &=-|\nabla^{\iota}H|^{2}.\label{eqn77}
\end{align}
\replaced{Now, by \eqref{eqn77} because of  $|H|=$ constant, we obtain that }{
If $|H|=$ constant. From \eqref{eqn77} it results that}
\begin{equation*}
    \nabla^{\iota}H=0.
\end{equation*}
Then, by Remark~\ref{Rq4.101}, we conclude that $M$ is minimal.
\end{proof}
\begin{thm}\label{Th4.55}
A compact Riemannian submanifold $\iota: (M^{m},g)\longrightarrow (N^{n},h)$ is $H$-tensional if and only if it is minimal.
\end{thm}
\begin{proof}
Assume that $M$ is $H$-tensional, by using the Weitzenb\"{o}ck formula \eqref{eqn77} and the divergence theorem, we obtain \added{that}
$\nabla^{\iota}H = 0$. But,
for all $X\in\Gamma(TM)$, we have \added{also that}
\begin{equation*}
    0=\nabla^{\iota}_{X}H=-A_{H}X+\nabla^{\bot}_{X}H.
\end{equation*}
Then, $\nabla^{\bot}_{X}H=0$ and \added{hence   $H = 0$ } by Remark \ref{Rq4.101}\deleted{we get $H = 0$}.
\end{proof}

\replaced{Corollary~\ref{Co1000} yields the following.}{From Corollary~\ref{Co1000}, we deduce}
\begin{thm}
There are no compact $H$-tensional submanifolds of the Euclidean space.
\end{thm}
\begin{thm}
A pseudo-umbilical submanifold $M$ of dimension $m$ with $m\neq 4$ is an $H$-tensional submanifold if and only if it is minimal.
\end{thm}
\begin{proof}
Let $\{e_i\}_{1\leq i\leq m}$ be a local geodesic orthonormal frame  at point $x$ in $M$. Calculating at $x$ and using \eqref{eq1001} \replaced{gives that }{we have}
\begin{equation}\label{eq1003}
    \sum_{i=1}^m[A_{\nabla_{e_i}^{\bot}H}e_i+\left[R^N(d\iota(e_i), H)d\iota(e_i) \right]^\top]=\sum_{i=1}^m\nabla_{e_i} A_H(e_i)-\frac{m}{2}\operatorname{grad}(|H|^{2}).
\end{equation}
Since $M$ is pseudo-umbilical, i.e., $M$ satisfying $A_{H}=|H|^{2}I$,  \added{the equation} \eqref{eq1003} becomes
\begin{equation}\label{eq1004}
    \sum_{i=1}^m[A_{\nabla_{e_i}^{\bot}H}e_i+\left[R^N(d\iota(e_i), H)d\iota(e_i) \right]^\top]=(1-\frac{m}{2})\operatorname{grad}(|H|^{2}).
\end{equation}
Now, replacing \eqref{eq1004} in the second equation of \eqref{eq12000}, we \replaced{obtain that}{get}
\begin{equation}\label{eq103}
    (4-m)\operatorname{grad}(|H|^{2})=0.
\end{equation}
Thus, \replaced{since}{for $m\neq 4$,} $|H|$ is constant\deleted{,} and \deleted{by}  Theorem~\ref{Th45} \added{valids   for $m\neq 4$,} we conclude that $M$ is minimal.
\end{proof}

\begin{thm}
A totally umbilical hypersurface in an Einstein space is $H$-tensional if and only if it is minimal.
\end{thm}

\begin{proof}
Let $\{e_{1},\cdots,e_{m},e_{m+1}\}$ be an orthonormal frame adapted to the
hypersurface $M$ so that $Ae_{i}=\lambda e_{i}$, where $\lambda$ is the common value of all principal normal curvatures at any point $x\in M$. Then
\begin{gather*}
\begin{aligned}
\alpha &= \frac{1}{m}\sum_{i=1}^m<Ae_{i},e_{i}>=\lambda,\;\; |A|^2 = m \lambda^2,\\
A(\operatorname{grad} \alpha) &= A\big(\sum_{i=1}^m e_{i}(\lambda)e_{i}\big)=\frac{1}{2}\operatorname{grad}\lambda^2.
\end{aligned}
\end{gather*}
The $H$-tensional equations \eqref{eq12 222} become
\begin{equation*}
\begin{cases}
\Delta \lambda + m\lambda^3
    =0,\\
(2+m)\operatorname{grad}\lambda^2= 0.
\end{cases}
\end{equation*}
\replaced{Hence}{So} $\lambda = 0$ \replaced{which gives that}{and hence}  $\alpha = 0$.
\end{proof}
\begin{cor}
Any totally umbilical $H$-tensional hypersurface in a Ricci flat manifold is minimal.
\end{cor}

\added{In the following, we completely follow  \cite{dimitric} and \cite{balmus}.}
\begin{thm}
Let $M$ be a hypersurface with at most two distinct principal curvatures in $N^{m+1}(c)$. Then $M$ is an $H$-tensional submanifold if and only if it is minimal.
\end{thm}
\begin{proof}
	\replaced{Assume that $M$ is an $H$-tensional hypersurface with at most two distinct principal curvatures in $N^{m+1}(c)$ which is not harmonic, i.e., the system \eqref{eq12 200} occurs. T}{We proceed as in \cite{dimitric} and \cite{balmus}. Assume that $M$ is $H$-tensional hypersurface with at most two distinct principal curvatures in $N^{m+1}(c)$ which is not harmonic i.e. the system \eqref{eq12 200} occurs, t}hen $H = \alpha e_{m+1}$ for some nonzero function $\alpha \in C^\infty(M)$ and some unit normal vector field $e_{m+1}$. Without loss of generality, we may assume that $\alpha > 0$. Let $U$ be an open set of $M$ defined by $U = \{ p \in M \mid (\operatorname{grad} \alpha^2)(p) \neq 0 \}$. We will prove that $U$ is empty. To do so, we assume that $U$ is nonempty. Let $\{e_i\}_{1\leq i\leq m}$ be the basis of principal directions on $U$ and $\{\omega_i\}_{1\leq i\leq m}$ its dual frame field so that $e_1 = \frac{\operatorname{grad} \alpha}{|\operatorname{grad} \alpha|}$. Then $e_1$ is a principal direction with associated principal curvature:
\[
\lambda_1 = -\dfrac{m}{2} \alpha.
\]
Now denote by $\omega^j_i$ to the connection $1$-forms given by $\nabla e_i = \omega^j_i e_j$. From the Codazzi equations for $M$ we get, for distinct $i, j, k = 1, \dots, m$, that
\begin{align}
e_j(\lambda_i) &= (\lambda_j - \lambda_i) \omega_j^i(e_i), \label{4.300} \\
(\lambda_j - \lambda_k) \omega_j^k(e_i) &= (\lambda_i - \lambda_k) \omega_i^k(e_j), \label{4.400}
\end{align}
where $\lambda_{i}$'s are principal curvatures. \added{Now, we have that}
\deleted{We have}
\begin{equation*}
    e_j(\alpha)=h(\operatorname{grad} \alpha,e_{j})=0,\;\;\;j=\overline{2,m},
\end{equation*}
\replaced{which implies that}{thus}
\begin{equation*}
    \operatorname{grad} \alpha=e_1(\alpha)e_1.
\end{equation*}
{\it Case 1.} If the multiplicity of $\lambda_{1}$ is at least $2$, i.e., if $\lambda_{i}=\lambda_{1}$ for some $i\geq 2$, then $e_1(\alpha)=0$. That follows
from \eqref{4.300} by considering $j = 1$ . This leads to $\operatorname{grad} \alpha=0$, a contradiction, so $U$ is empty and $\alpha$ is constant. Then,
by Theorem \ref{Th45} we conclude that $M$ is minimal.

\noindent {\it Case 2.} If the multiplicity of $\lambda_{1}$ is one, then as there are at most two distinct principal curvatures and \replaced{hence we obtain that }{since $\operatorname{Tr}A=m\alpha$ we have}
\begin{equation}\label{eqs101}
    \lambda_1 = -\dfrac{m}{2} \alpha, \;\;\lambda_2 =\lambda_3 =\cdots=\lambda_m = -\dfrac{3m}{2(m-1)} \alpha
\end{equation}
\added{since $\operatorname{Tr}A=m\alpha$.}
Therefore
\begin{equation}\label{eqs102}
    |A|^2=\lambda_1^{2}+\overset{m}{\underset{i=2}{\sum}}\lambda_i^{2} = \dfrac{m^{2}(m+8)}{4(m-1)} \alpha^{2}.
\end{equation}
\replaced{Now, s}{S}ince $\lambda_1\neq\lambda_j$ and $e_{j}(\alpha)=0$ for $j\geq 2$, we get from \eqref{4.300} \added{that}
\begin{equation}\label{eqs103}
\omega_1^j(e_1) = 0, \quad \text{for all}\; j = \overline{1,m},
\end{equation}
\replaced{i.e.,}{that is} $\nabla_{e_1}e_1=0$ which means that the integral curves of $e_{1}$ are geodesics on $U$. For $j=1$ and $i\geq 2$ \eqref{4.300}, gives
\begin{equation}\label{eqs104}
3 e_1(\alpha) = - (m + 2) \alpha \, \omega_1^i(e_i).
\end{equation}
For $i = 1$ and $j,k\geq 2$ with $j\neq k$, equation \eqref{4.400} leads to
\begin{equation}\label{eqs105}
\omega_1^k(e_j) = 0.
\end{equation}
From \eqref{eqs103}, \eqref{eqs104} and \eqref{eqs105}, we deduce that for $k\geq2$
\begin{equation}\label{eqs106}
-(m + 2)\alpha\omega_1^k= 3\alpha' \omega^k,
\end{equation}
where $'$ denotes derivative with respect to $e_{1}$. By using \eqref{eqs104}, the Laplacian of $\alpha$ can be computed as following
\begin{align}\label{eqs107}
\Delta \alpha &= \overset{m}{\underset{i=1}{\sum}}\left[(\nabla_{e_{i}}e_{i})\alpha-e_{i}(e_{i}(\alpha))\right]\notag\\
&= \overset{m}{\underset{i=1}{\sum}}\Big[\overset{m}{\underset{k=1}{\sum}}\omega_i^k(e_{i})(e_{k}(\alpha))-e_{i}(e_{i}(\alpha))\Big] \notag \\
&= \Big[\overset{m}{\underset{i=1}{\sum}}\omega_i^1(e_{i})\Big]\alpha'-\alpha'' \notag \\
&= \frac{3(m-1)}{(m + 2)\alpha}(\alpha')^{2}-\alpha''.
\end{align}
Thus, from \eqref{eqs102} and \eqref{eqs107}, the first equation of \eqref{eq12 200} becomes
\begin{equation}\label{eqss107}
\alpha''-\frac{3(m-1)}{(m + 2)\alpha}(\alpha')^{2}-\dfrac{m^{2}(m+8)}{4(m-1)} \alpha^{3}=0.
\end{equation}
On other hand, differentiating \eqref{eqs106}, we obtain \added{that}
\begin{equation}\label{eq1077}
 - (m + 2) \left(d\alpha \wedge \omega_1^k + \alpha\, d\omega_1^k\right)
= 3\, d\alpha' \wedge \omega^k + 3 \alpha'\, d\omega^k.
\end{equation}
The Cartan structural equations for $M$ are:
\begin{align}\label{eqs107}
d\omega_1^k = - \sum_{j=1}^m \omega_1^j \wedge \omega_k^j - (\lambda_1 \lambda_2 + c) \omega^1 \wedge \omega^k,
\end{align}
and
\begin{equation}\label{eq107}
d\omega^k=-\overset{m}{\underset{j=1}{\sum}} \omega^k_j \wedge \omega^j.
\end{equation}
Thus, from \eqref{eqs101}, \added{the equation} \eqref{eqs107} becomes
\begin{equation}\label{eqs108}
d\omega_1^k= - \sum_{j=1}^m \omega_1^j \wedge \omega_k^j+\Big(\frac{3m^2}{4(m - 1)} \alpha^2 - c\Big)\omega^1 \wedge \omega^k
\end{equation}
and, from \eqref{eqs106}, \added{the equation} \eqref{eq107} becomes
\begin{equation}\label{eq10777}
d\omega^k=\frac{3\alpha'}{(m +1)\alpha}\omega^k \wedge\omega^1 +\overset{m}{\underset{j=2}{\sum}} \omega^j_k \wedge \omega^j.
\end{equation}
Therefore
\begin{equation}\label{eqs109}
d\omega_1^k(e_1, e_k) = \frac{3m^2}{4(m - 1)} \alpha^2 - c,
\end{equation}
and
\begin{equation}\label{eqs110}
d\omega^k(e_1, e_k)=-\frac{3\alpha'}{(m +1)\alpha}.
\end{equation}
Moreover, we \replaced{obtain that}{yield}
\begin{equation}\label{eqs111}
d\alpha'\wedge\omega^k(e_1, e_k)=\alpha'',
\end{equation}
and
\begin{equation}\label{eqs112}
d\alpha\wedge\omega_1^k(e_1, e_k)=-\frac{3(\alpha')^{2}}{(m +1)\alpha}.
\end{equation}
Now, evaluating \eqref{eq1077} at $(e_1, e_k)$ and using equations \eqref{eqs109}-\eqref{eqs112}, we find \added{that}
\begin{equation}\label{eqs113}
\alpha'' - \frac{(m + 5)}{(m + 2)} \frac{(\alpha')^2}{\alpha}
+ \frac{m^2(m + 2)}{4(m - 1)} \alpha^3 - \frac{(m + 2)}{3} c\; \alpha = 0.
\end{equation}
Eliminating $\frac{(\alpha')^2}{\alpha}$ from \eqref{eqss107} and \eqref{eqs113} we find
\begin{equation}\label{eqss114}
(4-m)\alpha''=\frac{m^2(m^{2}+4m+17)}{4(m - 1)} \alpha^3 -\frac{(m^{2} + m-2) c}{2} \alpha .
\end{equation}

If $m=4$, \replaced{then}{from} \eqref{eqss114} \replaced{gives}{it follows} that $\alpha$ is constant, a contradiction\replaced{. Hence}{, so} $U$ is empty and $\alpha$ is constant. Then,
by Theorem~\ref{Th45}, we conclude that $M$ is minimal.

If $m\neq4$, multiply equation \eqref{eqss114} by $\alpha'$, integrate the result and obtain
\begin{equation}\label{eqs115}
(\alpha')^2=\frac{m^2(m^{2} + 4m+17)}{8(m - 1)(4-m)} \alpha^4
-\frac{(m^{2} + m-2) }{2(4-m)}c\; \alpha^{2}+ c_{0}.
\end{equation}
Eliminating $\alpha''$ from \eqref{eqss107} and \eqref{eqs113} we find
\begin{equation}\label{eqs116}
(\alpha')^2=\frac{m^2(m+5)(m+2)}{4(m - 1)(4-m)} \alpha^4
-\frac{(m+2)^{2} }{3(4-m)}c\; \alpha^{2}.
\end{equation}
We subtract \eqref{eqs116} from \eqref{eqs115}, we yield
\begin{equation}\label{eqs117}
\frac{m^2(m^{2}+10 m+3)}{8(m - 1)(m-4)} \alpha^4
+\frac{m^{2}-5 m-14}{6(m-4)}c\; \alpha^{2}+ c_{0}=0,
\end{equation}
so, by \eqref{eqs117}, it follows that $\alpha$ is constant, a contradiction, so $U$ is empty and $\alpha$ is constant. Then,
by Theorem \ref{Th45} we conclude that $M$ is minimal.
\end{proof}
\begin{cor}
Let $M$ be a surface of $N^{3}(c)$. Then $M$ is an $H$-tensional submanifold if and only if it is minimal.
\end{cor}
\begin{thm}
Let $\iota: (M^{m},g)\longrightarrow (N^{n},h)$ be an $H$-tensional non-compact complete Riemannian submanifold and let $q$ be a real constant satisfying $2\leq q<\infty.$ If
$$\int_{M}|H|^{q}v_{g}<\infty,$$
then $M$ is minimal.
\end{thm}
\begin{proof}
From Lemma \ref{lem100}, we have $\nabla^\iota _X H = 0$ for any $X\in\Gamma(TM)$ and $|H|$ is constant. By using Theorem \ref{Th45} we deduce that $M$ is minimal.
\end{proof}
\subsection{$H$-tensional curves}
In this last section, we focus on the simplest case of $H$-tensional maps.
\begin{thm}
Let $\gamma:I\longrightarrow(M^{m},g)$ be a regular curve parametrized by arc
length in a Riemannian manifold $(M^{m},g)$. Then, $\gamma$ is $H$-tensional curve if and only if it is a geodesic.
\end{thm}
\begin{proof}
Let $\gamma:I\longrightarrow(M^{m},g)$ be a regular curve parametrized by arc
length in a Riemannian manifold $(M^{m},g)$. Let $\{F_{i},i=\overline{1,m}\}$ be the Frenet frame in $M$ along $\gamma(s)$, which is obtained as
the orthonormalization of the $m$-tuple $\{\nabla_{\partial/\partial s}^{(k)}d\gamma(\partial/\partial s)|k=\overline{1,m}\}.$ Then we have the following Frenet formula (see \cite{laugwitz}) along the curve:
\begin{gather}
\begin{aligned}
  \nabla^{\gamma}_{\partial/\partial s}F_{1} &= \chi_{1}F_{2},\\
\nabla^{\gamma}_{\partial/\partial s}F_{i} &= -\chi_{i-1}F_{i-1}+\chi_{i}F_{i+1},\;\;\;\;i=\overline{2,m-1}\\
  \nabla^{\gamma}_{\partial/\partial s}F_{m} &=-\chi_{m-1}F_{m-1},
\end{aligned}
\end{gather}
where $\chi_{1},\chi_{2},\cdots,\chi_{m-1}$ are the curvatures of the curve $\gamma$. Using these Frenet equations, we obtain \added{that}
\begin{equation*}
    \tau(\gamma)=\nabla_{\gamma'}\gamma'=\chi_{1}F_{2},
\end{equation*}
where $\gamma'=\frac{d\gamma}{ds}$ and
\begin{gather}
\begin{aligned}
  \Delta^{\gamma}\tau(\gamma) &= -\nabla_{\gamma'}\nabla_{\gamma'}\nabla_{\gamma'}\gamma',\\
&= (3\chi_{1}\chi_{1}')F_{1}-(\chi_{1}''-\chi_{1}^{3}-\chi_{1}\chi_{2}^{2})F_{2}-(2\chi_{1}'\chi_{2}+\chi_{1}\chi_{1}')F_{3},\\
 &-(\chi_{1}\chi_{2}\chi_{3})F_{4}.
\end{aligned}
\end{gather}
Therefore, $\gamma$ is $H$-tensional curve if and only if
\begin{equation*}
\left\{
\begin{array}{ll}
\chi_{1}\chi_{1}'=0 , & \hbox{} \\
\chi_{1}''-\chi_{1}^{3}-\chi_{1}\chi_{2}^{2}=0 , & \hbox{} \\
2\chi_{1}'\chi_{2}+\chi_{1}\chi_{2}'=0 , & \hbox{} \\
\chi_{1}\chi_{2}\chi_{3}=0 , & \hbox{}%
\end{array}
\right.
\end{equation*}
\replaced{which implies}{it follows} that $\gamma$ is a nongeodesic $HS$-tensional curve\replaced{, i.e., it}{ that} is an $H$-tensional curve with $\chi_{1}\neq0$ if and only if
\begin{equation*}
\left\{
\begin{array}{ll}
\chi_{1}=\protect{constant}\neq0, & \hbox{} \\
\chi_{1}^{2}+\chi_{2}^{2}=0 , & \hbox{} \\
\chi_{2}=\protect{constant}, & \hbox{}\\
\chi_{1}\chi_{2}\chi_{3}=0 . & \hbox{}%
\end{array}
\right.
\end{equation*}
Thus we deduce that $\gamma$ is a geodesic.
\end{proof}

Based on the previous results on the nonexistence of nonminimal $H$-tensional submanifolds of real space form, \replaced{it seems highly probable that the followings  hold:\\}{we formulate the following}

\noindent\textbf{Conjecture 3}. The only $H$-tensional submanifolds of a real space form are the minimal ones. \\

\noindent\textbf{Conjecture 4}. The only complete $H$-tensional submanifolds of a real space form are the minimal ones.


\end{document}